\documentclass[11pt]{article}
\usepackage{amsmath,amssymb,amsthm}
\newcommand{\re}{\mathbb{R}}

\newcommand{\D}{\mathrm{D}}
\newcommand{\R}{\mathrm{R}}

\long\def\symbolfootnote[#1]#2{\begingroup%
\def\thefootnote{\fnsymbol{footnote}}\footnote[#1]{#2}\endgroup}

\newtheorem{thm}{Theorem}[section]

\newtheorem{lem}[thm]{Lemma}
\newtheorem{cor}[thm]{Corollary}

\theoremstyle{definition}

\theoremstyle{remark}

\newtheorem{rem}[thm]{Remark}
\usepackage{comment}

\title{On local solvability of nonlinear elliptic partial differential systems of principle type: the second order }

\author{Yifei Pan}

\begin{document}

\maketitle
\begin{center}
\end{center}
\begin{abstract}
We prove results on solvability of nonlinear elliptic partial differential systems of principle type of second order. They are consequences of existence of non-radial solutions for nonlinear partial differential systems of Poisson type. As applications to geometry, we prove the exsitence of local harmonic maps with given tangent  plane at a point between any Riemannan manifolds. More generally geometric objects defined by Beltrami-Laplace always exist locally.

\end{abstract}

\large

\section{Introduction}\label{sec0}\symbolfootnote[0]{MSC 2010: 35G20 (Primary); 32G05, 30G20 (Secondary)}

In this paper we consider solvability problem for nonlinear partial differential systems of principle type of second order in $\re^n$. To this end,  we first obtain existence of non-radial classical solutions of nonlinear systems of Poisson type under rather general conditions. It implies in particular the following general results on solvability of nonlinear elliptic partial differential systems of second order.
\subsection{General elliptic systems of second order}
In this paper we denote $L$ as an elliptic operator; namely,
$$L u=a^{ij}(x) D_{ij} u$$
where we assume that $a^{ij}(x)\in C^{1,\alpha}$ and satisfies 
$$a^{ij}(x)\xi_i\xi_j\geq \lambda|\xi|^2 $$
for some positive constant $\lambda$ and for all $\xi\in\re^n$.
Our first result is that the elliptic system is always solvable with any initial value at a point.

\bigskip
\noindent
{\bf Theorem A.} {\sl Let $a(x, p, q)=(a_1(x, p, q), ..., a_N(x, p, q))$ be of class $C_{loc}^{k+\alpha}$ ($1\leq k\leq\infty, 0<\alpha<1)$,  where $x\in\re^n, p\in\re^N$, and $q\in \re^n\otimes\re^N$.
Then, for any given $c_0\in \re^N, c_1\in \re^n\otimes\re^N $, the following system: $u(x)=(u_1(x), ..., u_N(x)): \{|x|\leq R\}\to \re^N$,
\begin{eqnarray}
L u(x)&=&a(x, u(x), \nabla u(x))\nonumber\\
u(0)&=&c_0\nonumber\\
\nabla u(0)&=&c_1\nonumber
\end{eqnarray}
has infinitely many solutions of $C^{k+2+\alpha}(\{|x|\leq R\})$ for sufficiently small values of $R$. In particular, all solutions are not radially symmetric and satisfy $\nabla^2 u(0)\not=0$.}

\bigskip
The second result is to show the solvability of fully nonlinear systems.

\bigskip
\noindent
{\bf Theorem B.} {\sl
Let $a(x, p, q, r)=(a_1(x, p, q, r), ..., a_N(x, p, q, r))$ be of class $C_{loc}^{2}$ ($ 0<\alpha<1)$,  where $x\in\re^n, p\in\re^N, q\in \re^n\otimes\re^N$, and $r\in \mathrm{Sym}(n)\otimes\re^N$.
Assume that
$a(0)=\nabla_r a(0)=\nabla_r^2 a(0)=0.$
Then the following system: $u(x)=(u_1(x), ..., u_N(x)): \{|x|\leq R\}\to \re^N$,
\begin{eqnarray}
L u(x)&=&a(x, u(x), \nabla u(x), \nabla^2 u(x))\nonumber
\end{eqnarray}
has infinitely many solutions of $C^{2+\alpha}(\{|x|\leq R\})$ of vanishing order two at the origin for sufficiently small values of $R$. Moreover, all these solutions are not radially symmetric.
}

\bigskip
The third result is to show the existence of semi-global solutions if the system is autonomous.

\bigskip
\noindent
{\bf Theorem C.} {\sl
 Let $a(p, q, r)=(a_1(p, q, r), ..., a_N(p, q, r))$ be of class $C_{loc}^{2}$ ($0<\alpha<1)$,  where $p\in\re^N, q\in \re^n\otimes\re^N$, and $r\in \mathrm{Sym}(n)\otimes\re^N$.
Assume
$a(0)=
\nabla a(0)=0.$
Then the following system: $u(x)=(u_1(x), ..., u_N(x)): \{|x|\leq R\}\to \re^N$,
\begin{eqnarray}
L u(x)&=&a(u(x), \nabla u(x), \nabla^2 u(x))\nonumber
\end{eqnarray}
has infinitely many solutions in $\{|x|\leq R| x\in\re^n\}$ of $C^{2+\alpha}$ with vanishing order two at the origin for any given value of $R$.
Consequently, all these solutions are not radially symmetric.
}

\bigskip
As applications to differential geometry, we prove local existence of harmonic maps.

\bigskip
\noindent
{\bf Theorem D.} {\sl
Let $M, N$ be two Riemanian manifolds. Let $p\in M, q\in N$ be any points. Then there is a local harmonic map between $M$ and $N$ near $p$ such that its tangent space at $q$ can be arbitrarily given.}

\bigskip
We point out that the existence of nontrival solutions of elliptic system implies that geometric equations or systems with Laplace-Beltrami on a Riemannian manifold are always locally solvable. For example we use Theorem A to give a simple proof of the existence of harmonic coordinates at a point of a Riemannian manifold,
which is of course well-known. Given a Riemannian manifold $(M,g)$, we can think of $M$ being open subset of $\re^n$ and $p=0$. Then the metric is given by
$$g=g_{ij}=g(\partial_i,\partial_j)=g(\frac{\partial}{\partial y^i},\frac{\partial}{\partial y^j})$$
in the standard Cartesian coordinates $(y^1, ...,y^n)$. We can find a coordinate transformation $y\to x$ by solving the system
\begin{eqnarray}
\Delta x^k &=&\frac{1}{\sqrt{\det g_{ij}}}\partial_i(\sqrt{\det g_{ij}}g^{ij}\partial_j x^k)=0\nonumber\\
x^k(0) &=&0\nonumber\\
\nabla x^k(0)&=& e_k.\nonumber
\end{eqnarray}
Evidently, $x=(x^1,...,x^n)$ becomes so-called harmonic coordinates at $p=0$.

Let us take a look at the equation of prescribed mean curvature. Any local solution by Theorem A of the following
$$(1+|D u|^2)\Delta u-D_iuD_j uD_{ij} u=nH(x)(1+|D u|)^2)^{3/2}$$
gives a graph $(x,u(x))$ in $\re^{n+1}$ whose mean curvature is $H(x)$ at $(x,u(x))$.

All these results are consequences of systems of Poisson type that we present specifically below.

\subsection{Systems of Poisson type}

\begin{thm} Let $a(x, p, q, r)=(a_1(x, p, q, r), ..., a_N(x, p, q, r))$ be of class $C_{loc}^{2}$ ($ 0<\alpha<1)$,  where $x\in\re^n, p\in\re^N, q\in \re^n\otimes\re^N$, and $r\in \mathrm{Sym}(n)\otimes\re^N$.
Assume that
\begin{eqnarray}
a(0)&=& 0,\\
\nabla_r a(0)&=&0,\\
\nabla_r^2 a(0)&=&0.
\end{eqnarray}
Then the following system: $u(x)=(u_1(x), ..., u_N(x)): \{|x|\leq R\}\to \re^N$,
\begin{eqnarray}
\Delta u(x)&=&a(x, u(x), \nabla u(x), \nabla^2 u(x))
\end{eqnarray}
has infinitely many solutions of $C^{2+\alpha}(\{|x|\leq R\})$ of vanishing order two at the origin for sufficiently small values of $R$. Moreover, all these solutions are not radially symmetric.
\end{thm}
Here we make comments on notations used above. First the notation $\nabla_r$ means that derivatives are taken with respect to the variables of $r$, and $\mathrm{Sym}(n)$ denotes the set of $n\times n$ symmetric matrices. A solution $u$ is said of vanishing order two if $u(0)=0, \nabla u(0)=0$, but $\nabla^2 u(0)\not=0$. The vanishing order two ensures that the solutions obtained are non-trivial; namely they are neither constant nor linear ones.
More importantly, the vanishing order two ensures that solutions of such are not radial, and this shows that the solutions could not come from solving an ODE.

If $a$ is independent of $r$, then the following existence theorem can be regarded as a variant of ODE with initial values. We point out, though,  that there is no uniqueness result as in ODE; instead there are infinitely many solutions by construction.
\begin{thm} Let $a(x, p, q)=(a_1(x, p, q), ..., a_N(x, p, q))$ be of class $C_{loc}^{k+\alpha}$ ($1\leq k\leq\infty, 0<\alpha<1)$,  where $x\in\re^n, p\in\re^N$, and $q\in \re^n\otimes\re^N$.
Then, for any given $c_0\in \re^N, c_1\in \re^n\otimes\re^N $, the following system: $u(x)=(u_1(x), ..., u_N(x)): \{|x|\leq R\}\to \re^N$,
\begin{eqnarray}
\Delta u(x)&=&a(x, u(x), \nabla u(x))\\
u(0)&=&c_0\\
\nabla u(0)&=&c_1
\end{eqnarray}
has infinitely many solutions of $C^{k+2+\alpha}(\{|x|\leq R\})$ for sufficiently small values of $R$. In particular, all hese solutions are not radially symmetric.
\end{thm}

If $a$ is independent of $x$, that is, the system is so-called autonomous, then we can solve non-trivial semi-global solutions, i.e., solutions that are defined in any given ball $\{|x|\leq R\}$.
\begin{thm} Let $a(p, q, r)=(a_1(p, q, r), ..., a_N(p, q, r))$ be of class $C_{loc}^{2}$ ($0<\alpha<1)$,  where $p\in\re^N, q\in \re^n\otimes\re^N$, and $r\in \mathrm{Sym}(n)\otimes\re^N$.
Assume
\begin{eqnarray}
a(0)&=& 0,\\
\nabla a(0)&=&0.
\end{eqnarray}
Then the following system: $u(x)=(u_1(x), ..., u_N(x)): \{|x|\leq R\}\to \re^N$,
\begin{eqnarray}
\Delta u(x)&=&a(u(x), \nabla u(x), \nabla^2 u(x))
\end{eqnarray}
has infinitely many solutions in $\{|x|\leq R| x\in\re^n\}$ of $C^{2+\alpha}$ with vanishing order two at the origin for any given value of $R$.
Consequently, all these solutions are not radially symmetric.
\end{thm}

We remark that the conditions (8) and (9) imply that the nonlinearity of $a$ does not contain linear terms and that $u\equiv 0$ is a trivial solution of (10). These conditions are necessary due to a well-know result
of Osserman [O]. In fact, let us solve the scalar equation
$$\Delta u=e^{2u}$$
with the initial value $u(0)=a$, which exists for a small $R$ according to Theorem 1.2. On other hand, by Osserman's theorem applied to this case,
we have
$$R\leq \frac{1}{e^{u(0)}}=e^{-a}.$$
Letting $a\to +\infty$, we see that $R\to 0$. Of course, $a(p)=e^p$ does not satisfies the condition (8). This example also shows that $R$ needs to be small in general. Another example, we may consider, is the eigenvalue equation
$$\Delta u=\lambda u,$$
which, of course, has no non-zero solutions for most values of $\lambda$.

A classical and well-known result of Gidas-Ni-Nirenberg [GNN] says that if $u$ is a positive solution of the Poisson equation
$$\Delta u=f(u)$$
in $|x|<R$
with $f$ $C^1$ and
$$u|_{\{|x|=R\}}=0,$$
then $u$ is radial. However their theorem does not prove the existence of such a solution. As an application of our results,
we prove the following corollary for existence of non-radial solutions.
\begin{cor} Let $f(t)$ be a function of class $C^{k+\alpha}(\re)$ ($k\geq 1, 0<\alpha<1$). Assume that $f(0)=f'(0)=0$. Then the equation
$$\Delta u=f(u)$$
has, for any given $R$,  infinitely many solutions of class $C^{k+2+\alpha}(\{|x|\leq R\})$ which are of vanishing order $2$ at the origin and  are neither radially symmetric nor positive.
\end{cor}
For example, much study has been conducted on seeking positive solutions of the equation
$$\Delta u=c|u|^{\frac{n+2}{n-2}}$$
in $\re^n$. Corollary 1.4, however,  implies that there are indeed infinitely many non-radial and non-positive solutions for a ball of any radius. However it does not conclude the existence in the whole $\re^n$, which is in general impossible due to [O].

We would like to point out that for solvability of linear partial differential equations, there is a well-known so-called Nirenberg-Treves conjecture. This conjecture was recently solved by Dencker [D], following previous important works in [L], [NT], [H], [BF], [LE]. Our consideration on nonlinear cases is very different from linear ones technically.

In the paper [P1] dealing with dimension two, complex analysis allows us to prove similar results with power of Laplace. In higher dimension, we carry out the method of this paper and that of [P1] to study the general system of higher order 
$$\Delta^m u(x)=a(x,u,\nabla u,..., \nabla^m u).$$
in a joint paper [PZ].

This paper is organized as follows. First we introduce a Banach space with vanishing order from which we are seeking possible solutions. Then, we study Newtonian potential as an operator on the Banach space, which seems to have been overlooked in the literature. Finally, we construct an operator map from which we will try to produce a fixed point on a closed subset of the Banach space. A large portion of the paper is on estimating H\" older norm of involved functions in order to apply Fixed point theorem.

\section{Function spaces and their norms}

In this paper throughout, we let $D$ denote the closed ball $\{x\in\re^n\mid|x|\leq R\}$ and $C$ its boundary $\{x\in\re^n \mid |x|=\R\}$. Unless otherwise stated, all functions considered will be real-valued and integrable, with domain $D$. We will consider some classes of functions.

\subsection{H\" older space}
$C^{\alpha}(D)$ is the set of all functions $f$ on $D$ for which
$$ H_{\alpha}[f]=\sup\bigg\{{\frac{|f(x)-f(x')|}{|x-x'|^{\alpha}} \bigg| x,x' \in \D} \bigg\}$$
is finite.

$C^{k}(D)$ is the set of all functions $f$ on $D$ whose $k^{\textup{th}}$ order partial derivatives exist and are continuous.
$C^{k+\alpha}(D)$ is the set of all functions $f$ on $D$ whose $k^{\textup{th}}$ order partial derivatives exist and belong to $C^{\alpha}(D)$.

The symbol $|f|$ or $|f|_\D$ denotes $\textup{sup}_{x\in D}|f(x)|$.
For $f\in C^\alpha(D)$ we define the norm
$$\|f\|=|f|+(2R)^\alpha H_\alpha[f].$$
The set of $N$-tuples $f=(f_1,..., f_N)$ of functions (vector functions or maps) of $C^\alpha (D)$ is denoted by $[C^\alpha (D)]^N$, and $H_\alpha[f]$ is defined as the maximum of $H_\alpha[f_i](i=1,..,N)$. In a similar fashion we define $|f|_A=\sup_{x\in A}|f(x)|$ for functions and vector functions, and write $|f|$ when the domain is understood. Finally, in this paper throughout, the norm of $\re^N$ is taken as $|v|=\max_{1\leq j\leq N}|v_j|$.

The following lemma is well-known; see ([GT]).
\begin{lem}
The function $\|\cdot\cdot\cdot\|$ defined on $C^{\alpha}(D)$ is a norm, with respect to which $C^{\alpha}(D)$ is a Banach algebra: $\|fg\|\leq \|f\|\|g\|$.
\end{lem}

\subsection{Function spaces with vanishing order at the origin}
Our idea of solving differential equations or systems of order $m$ is to look for solutions that vanish up to $m-1$ order at the origin; this way the norm estimate of the function space to be considered later is made possible in terms of only $m$th order derivatives. This is rather different from classical norms used for higher order derivatives in partial differential equations. 

We denote for $k\geq 1$, $C_0^{k+\alpha}(D)$ the set of all functions in $C^{k+\alpha}(D)$ whose derivatives vanish up to order $k-1$ at the origin. Specifically
$$C_0^{k+\alpha}(D)=\{f\in  C^{k+\alpha}(D)\big | \partial^\beta f(0)=0, |\beta|\leq k-1\},$$
where we have used $\beta=(\beta_1,...,\beta_n)$ and $|\beta|=\beta_1+...+\beta_n$. Also we have
$$\partial^\beta=\partial_1^{\beta_1}\cdot\cdot\cdot\partial_n^{\beta_n}, \mbox{ }
\partial_{ij}=\partial_i\partial_j.$$
One has the following obvious nesting
$$C_0^{m+\alpha}(D)\subset C_0^{m-1+\alpha}(D)\subset\cdot\cdot\cdot\subset C_0^{1+\alpha}(D)\subset C^{\alpha}(D).$$
We now define a function $\|\cdot\cdot\cdot\|^{(k)}$ on $C^{k+\alpha}(D)$:
$$\|f\|^{(k)}=\max_{ |\beta|=k}\{\|\partial^\beta f\|\}.$$
We point out that the function $\|\cdot\cdot\cdot\|^{(k)}$ on $C^{k+\alpha}(D)$ is not a norm since $\|f\|^{(k)}=0$ if and only if $f$ is
a polynomial of degree at most $k-1$. However it becomes norm when restricted to the subspace $C_0^{k+\alpha}(D)$, which is to be proved below and is one of important facts used in
this paper. First we obtain some useful estimates, which will be used repeatedly later.
\begin{lem} If $f\in C^{k+\alpha}(D)$, then, for $x, x'\in D$,
$$|f(x')-\sum_{l=0}^k\frac{1}{l!}\sum_{|\beta|=l}\partial^\beta f(x)(x'-x)^\beta|\leq \frac{1}{k!}\bigg\{\sum_{|\beta|=k}H_\alpha[\partial^\beta f]\bigg\}|x'-x|^{k+\alpha}.$$
\end{lem}
\begin{proof}
Expanding at $x$, we have the formula
\begin{eqnarray*}
&&f(x')-\sum_{l=0}^{k-1}\frac{1}{l!}\sum_{|\beta|=l}\partial^\beta f(x)(x'-x)^\beta\\
&=&\int_0^1\int_0^{t_{k-1}}\cdot\cdot\cdot\int_0^{t_1}\bigg\{\frac{d^k}{dt^k}f(t x'+(1-t)x)\bigg\}dtdt_1\cdot\cdot\cdot dt_{k-1}\nonumber\\
&=&\int_0^1\int_0^{t_{k-1}}\cdot\cdot\cdot\int_0^{t_1}\bigg\{\sum_{|\beta|=k}\partial^\beta f(tx'+(1-t)x)(x'-x)^\beta\bigg\}dtdt_1\cdot\cdot\cdot dt_{k-1}.\nonumber
\end{eqnarray*}
Hence, we have, by subtracting kth term,
$$f(x')-\sum_{l=0}^k\frac{1}{l!}\sum_{|\beta|=l}\partial^\beta f(x)(x'-x)^\beta$$
$$=\int_0^1\int_0^{t_{k-1}}\cdot\cdot\cdot\int_0^{t_1}\bigg\{\sum_{|\beta|=k}\{\partial^\beta f(tx'+(1-t)x)-\partial^\beta f(x)\}(x'-x)^\beta\bigg\}dtdt_1\cdot\cdot\cdot dt_{k-1}.$$
Thus we have,
$$|f(x')-\sum_{l=0}^k\frac{1}{l!}\sum_{|\beta|=l}\partial^\beta f(x)(x'-x)^\beta|$$
$$\leq \frac{1}{k!}\sum_{|\beta|=k}H_\alpha[\partial^\beta f]|x'-x|^{k+\alpha}.$$
This completes the proof.
\end{proof}
\begin{lem}
If $f\in C_0^{k+\alpha}(D)$, then
$$\|f\|\leq \frac{(3n)^k}{k!}R^k\|f\|^{(k)}.$$
\end{lem}
\begin{proof}
Let $f\in C_0^{k+\alpha}(D)$, then
\begin{eqnarray*}
f(x)&=&\int_0^1\int_0^{t_{k-1}}\cdot\cdot\cdot\int_0^{t_1}\bigg\{\frac{d^k}{dt^k}f(t x)\bigg\}dtdt_1\cdot\cdot\cdot dt_{k-1}\nonumber\\
 &=&\int_0^1\int_0^{t_{k-1}}\cdot\cdot\cdot\int_0^{t_1}\bigg\{\sum_{|\beta|=k}\partial^\beta f(t x)x^\beta\bigg\}dtdt_1\cdot\cdot\cdot dt_{k-1}\nonumber\\
&=&\sum_{|\beta|=k}\bigg\{\int_0^1\int_0^{t_{k-1}}\cdot\cdot\cdot\int_0^{t_1}\partial^\beta f(t x) dtdt_1\cdot\cdot\cdot dt_{k-1}\bigg\}x^\beta.\nonumber
\end{eqnarray*}
Applying norm inequality, we obtain
\begin{eqnarray*}
\|f\|&\leq & \sum_{|\beta|=k}\frac{1}{k!}\|\partial^\beta f\|\|x^\beta\|\nonumber\\
&\leq &\sum_{|\beta|=k}\frac{1}{k!}\|\partial^\beta f\|\|x\|^k\leq\frac{n^k}{k!}(3R)^k\|f\|^{(k)}\nonumber,
\end{eqnarray*}
where we have used $\|x_i\|=3R$, which is easily verified.
\end{proof}
\begin{lem}
If $f\in C_0^{k+\alpha}(D)$, then, for $|\beta|\leq k$,
$$\|\partial ^\beta f\|\leq \frac{(3n)^{k-|\beta|}}{(k-|\beta|)!}R^{k-|\beta|}\|f\|^{(k)}.$$
\end{lem}
\begin{proof}
Let  $f\in C_0^{k+\alpha}(D)$. If $|\beta|\leq k$, then $\partial^\beta f\in C_0^{k-|\beta|+\alpha}(D)$. By Lemma 2.3, we have
$$\|\partial^\beta f\|\leq \frac{(3n)^{k-|\beta|}}{(k-|\beta|)!}R^{k-|\beta|}\|\partial^\beta f\|^{(k-|\beta|)}\leq \frac{(3n)^{k-|\beta|}}{(k-|\beta|)!}R^{k-|\beta|}\| f\|^{(k)}.$$
\end{proof}
An immediate corollary is the following:
\begin{lem}
If $f\in C_0^{k+\alpha}(D)$, then, for $l\leq k$,
$$\|f\|^{(l)}\leq \frac{(3n)^{k-l}}{(m-l)!}R^{k-l}\|f\|^{(k)}.$$
\end{lem}

In order to verify that $C_0^{k+\alpha}(D)$ is a Banach space with norm $\|\cdot\cdot\cdot\|^{(k)}$, we need the following simple lemma.
\begin{lem}
Let $\{f_m\}_{m=1}^\infty$ be a sequence in $C^\alpha(D)$, with $\|f_m\|\leq M$; assume that $\{f_m\}$ converges to a function $f$ at each point of $D$. Then $f\in C^\alpha(D)$, $\|f\|\leq M$.
\end{lem}
\begin{proof} Since $|f_m|+(2R)^\alpha H_\alpha(f_m)\leq \|f_m\|\leq M$, we also have
$$|f_m(x)-f_m(x')|\leq (2R)^{-\alpha}M|x-x'|^\alpha$$
for all $x, x'\in D$. Letting $m\to\infty$ in the above inequality, we conclude
$$|f(x)-f(x')|\leq (2R)^{-\alpha}M|x-x'|^\alpha,$$
and
$$|f|+(2R)^\alpha H_\alpha(f)\leq M.$$
Therefore $\|f\|\leq M$ .
\end{proof}
\begin{lem}
Let $\{f_m\}$ be a sequence in $C_0^{k+\alpha}(D)$, with $\|f_m\|^{(k)}\leq M$, and if $\{\partial^\beta f_m\}$, for all $\beta, |\beta|=k$, are Cauchy sequences in the norm $|\cdot\cdot\cdot|$, then there is a function $f\in C_0^{k+\alpha}(D)$ such that $|\partial^\beta f_m-\partial^\beta f|\to 0$ as $m\to\infty$ for all $\beta, 0\leq |\beta|\leq k$, and with $\|f\|^{(k)}\leq M$.
\end{lem}
\begin{proof}
For $\beta, |\beta|=k$, consider $g_m^\beta=\partial^\beta f_m$, then $g_m^\beta\in C^\alpha(D)$ and $\|g_m^\beta\|\leq M$. Applying to Lemma 2.6, we
have functions $g^{\beta}\in C^\alpha(D)$ such that $|g_m^{\beta}-g^{\beta}|\to 0$ as $ m \to \infty$, with $\|g^{\beta}\|\leq M $.
Define $f$ by
$$f(x)=\int_0^1\int_0^{t_{k-1}}\cdot\cdot\cdot\int_0^{t_1}\bigg\{\sum_{|\beta|=k}g^{\beta} (t x)x^\beta \bigg\} dtdt_1\cdot\cdot\cdot dt_{k-1}.$$
Since $f\in C_0^{k+\alpha}(D)$, we have
$$f_m(x)=\sum_{|\beta|=k}\bigg\{\int_0^1\int_0^{t_{k-1}}\cdot\cdot\cdot\int_0^{t_1}\partial^\beta f_m(t x) dtdt_1\cdot\cdot\cdot dt_{k-1}\bigg\}x^\beta$$
Therefore
$$f_m(x)-f(x)=\sum_{|\beta|=k}\bigg\{\int_0^1\int_0^{t_{k-1}}\cdot\cdot\cdot\int_0^{t_1}\bigg\{\partial^\beta f_m(t x)-g^{\beta}(t x)\bigg\} dtdt_1\cdot\cdot\cdot dt_{k-1}\bigg\}x^\beta,$$
whence
$$|f_m-f|\leq \frac{R^k}{k!}\sum_{|\beta|=k}|\partial^\beta f_m-g^{\beta}|,$$
which goes to $0$ as $m\to\infty$, implying $f_m\to f$ in the norm $|\cdot\cdot\cdot|$.

For $|\gamma|=l\leq k-1$ we want to prove that $\{\partial^\gamma f_m(x)\}$ are Cauchy sequences. Indeed,
Since $f_m$ vanishes up to $ k-1$ order at the origin, then $\partial ^\gamma f_m$ vanishes to
$k-1-l$ order at the origin. Thus, we have the formula,
\begin{eqnarray}
\partial^\gamma f_m(x)&=&\int_0^1\cdot\cdot\cdot\int_0^{t_{k-1-l}}\frac{d^{k-l}}{dt^{k-l}}\partial^\gamma f_m(tx)dt\cdot\cdot\cdot dt_{k-1-l}\nonumber\\
&=&\int_0^1\cdot\cdot\cdot\int_0^{t_{k-1-l}}\sum_{|\beta|=k-l}\partial^\beta\partial^\gamma f_m(tx)x^\beta dt\cdot\cdot\cdot dt_{k-1-l}\nonumber\\
&=&\sum_{|\beta|=k-l}\int_0^1\cdot\cdot\cdot\int_0^{t_{k-1-l}}\partial^\beta\partial^\gamma f_m(tx)dt\cdot\cdot\cdot dt_{k-1-l}x^\beta.
\end{eqnarray}
Then
\begin{eqnarray}
&&\partial ^\gamma f_m(x)-\partial ^\gamma f_{m'}(x)\nonumber\\
&=&\sum_{|\beta|=k-l}\int_0^1\cdot\cdot\cdot\int_0^{t_{k-1-l}}\partial^\beta\partial^\gamma \{f_m(tx)-f_{m'}(tx)\}dt\cdot\cdot\cdot dt_{k-1-l}x^\beta,
\end{eqnarray}
Then
\begin{eqnarray}
|\partial ^\gamma f_m(x)-\partial^\gamma f_{m'}(x')|\leq\frac{R^{k-l}}{(k-l)!}\sum_{|\beta|=k}|\partial^\beta f_m-\partial^\beta f_{m'}|,
\end{eqnarray}
which proves that $\{\partial^\gamma f_m(x)\} (|\gamma|\leq k-1)$ are Cauchy sequences since $\{\partial^\beta f_m(x)\} (|\beta|=k)$ are.
We assume that for $|\beta|=l\leq k-1$, $\partial ^\beta f_m(x)$ converges to $g^\beta$ in norm $|\cdot\cdot\cdot|$.
Thus, applying Lemma 2.2,  we have
\begin{eqnarray*}
&&\bigg|f_m(x')-\sum_{l=0}^k\frac{1}{l!}\sum_{|\beta|=l}\partial^\beta f_m(x)(x'-x)^\beta\bigg|\nonumber\\
&\leq& \frac{1}{k!}\sum_{|\beta|=k}H_\alpha[\partial^\beta f_m]|x'-x|^{k+\alpha}\nonumber\\
&\leq& \frac{n^k}{k!}(2R)^{-\alpha}\|f_m\|^{(k)}|x'-x|^{k+\alpha}
\leq \frac{2^k}{k!}(2R)^{-\alpha}M|x'-x|^{k+\alpha},\nonumber
\end{eqnarray*}
which is independent of $m$. Letting $m\to \infty$,
we have
$$|f(x')-f(x)-\sum_{l=1}^k\frac{1}{l!}\sum_{|\beta|=l}g^{\beta}(x'-x)^\beta|\leq \frac{2^k}{k!}(2R)^{-\alpha}M|x'-x|^{k+\alpha},$$
which implies, by definition of differentiability, $g^{\beta}=\partial^\beta f$. This implies $\|f\|^{(k)}\leq M$  by taking limit
from $\|f_m\|^{(k)}\leq M$. The convergence for $|\beta|\leq k-1$, follows from that of $|\beta|=k$ by (12).
\end{proof}
\begin{lem}
The function space $C_0^{k+\alpha}(D)$ equipped with the function $\|\cdot\cdot\cdot\|^{(k)}$ is a Banach space.
\end{lem}
\begin{proof}
Let $\{f_m\}$ be a Cauchy sequence in $\|\cdot\cdot\cdot\|^{(k)}$. Then for any $\epsilon>0$, there is $m_0$ so that if $m, m'>m_0$ it holds
\begin{eqnarray}
\|f_m-f_{m'}\|^{(k)}<\epsilon,
\end{eqnarray}
which implies
$$|\|f_m\|^{(k)}-\|f_{m'}\|^{(k)}|<\epsilon,$$
which implies $\{\|f_m\|^{(k)}\}$ are Cauchy sequence and therefore bounded by say $M$. Also by definition of $\|\cdot\cdot\cdot\|^{(k)}$, (14) implies $\partial^\beta f_m$ for
$|\beta|=k$ are Cauchy sequences in $|\cdot\cdot\cdot|$. By Lemma 2.7, there is a function $f\in C^{k+\alpha}_0(D)$ such that
$|\partial^\beta f_m-\partial^\beta f|\to 0$ for $|\beta|=k$ as $m\to\infty$, and $\|f\|^{(k)}\leq M$.

Now the sequence $\{f_m-f_{m'}\}_{m'=m+1}^\infty$ is bounded in $\|\cdot\cdot\cdot\|^{(k)}$ by $\epsilon$, and converges to $f_m-f$, with
$$\partial^\beta(f_m-f_{m'})\to \partial^\beta(f_m-f)$$
in $|\cdot\cdot\cdot|$ as $m'\to\infty$, $|\beta|=k$. By Lemma 2.7 again, it holds
$$\|f_m-f\|^{(k)}<\epsilon,$$
which implies that $f_m\to f$ in $\|\cdot\cdot\cdot\|^{(k)}$. The proof is complete.
\end{proof}
\section{Newtonian potentials and H\"older estimates}
\subsection{Definitions and basic facts}
The fundamental solution of Laplace's equation is given by
$$\Gamma(x-y)=\frac{1}{n(2-n)\omega_n}|x-y|^{2-n}$$
for $n>2$,
$$\Gamma(x-y)=\frac{1}{2\pi}\ln|x-y|$$
for $n=2$.
For an integrable function $f$ on $D$, the Newtonian potential of $f$ is the function $\mathcal{N}(f)(x)$ defined on $\re^n$ by
$$\mathcal{N}(f)(x)=\int_D \Gamma(x-y)f(y)dy.$$
In this paper, our point view is to consider $\mathcal{N}$ as an operator acting on function space. The following is well-known.
\begin{lem} Let $f\in C^\alpha(D)$. Then it holds, for $x\in \mathrm{Int}(D)$,
$$\partial_{ij}\mathcal{N}(f)(x)=\int_D \partial_{ij}\Gamma(x-y)(f(y)-f(x))dy-\frac{\delta_{ij}}{n}f(x)$$
\end{lem}
The proof of this lemma can be found in [F, p.204] (also also [GT]) and  is based on the following (not so obvious) integrals:
$$\mathcal{N}(1)(x)=-\frac{1}{2n}|x|^2+\frac{1}{2(n-2)}R^2$$
for $n\geq 3$, $x\in \mathrm{Int}(D)$ and
$$\mathcal{N}(1)(x)=-\frac{1}{4}|x|^2+\frac{1}{2}R^2(\log \frac{1}{R}+\frac{1}{2})$$
for $n=2$, $x\in \mathrm{Int}(D)$.

Now we discuss a technical result in order to deal with H\" older estimate of functions for operator $\mathcal{N}$. Let $x$ be an interior point of $D$, and $D_0$ the intersection of $D$ with the open ball of radius $\rho$ and center $x$. Namely
$$D_0=D\cap\{y\in\re^n||y-x|<\rho\}.$$

The following is essentially proved in [F] and fundamentally important in our approach that follows.
\begin{lem}  There is a constant $C$ only depending on $n$ (independent of $R$ and $\rho$) such that for $1\leq i, j\leq n$
$$\bigg|\int_{D\setminus D_0} \partial_{ij} \Gamma(x-y)dy\bigg|\leq C, x\in \mathrm{Int}(D).$$
\end{lem}
\begin{rem} In fact, Lemma A.14 in [F, p. 207] assumes that $\rho\leq \frac{R}{4}$. But the proof would go through essentially without the restriction. The reason is that Remark 4 in [F, p. 209] suggests in order to estimate H\" older constant it suffices to consider the case $|x-x'|<\frac{R}{4}$.
Otherwise when $|x-x'|\geq \frac{R}{4}$, one would have
$$\frac{|f(x)-f(x')|}{|x-x'|^\alpha}\leq \frac{2\sup|f(x)|}{(\frac{R}{4})^\alpha}.$$
But this estimate would be very bad for our consideration in which we would take $R\to 0$.
\end{rem}

We note Lemma 4.4 in [GT] states that $\mathcal{N}$ maps $C^\alpha(\frac{1}{2}D)$ to $C^\alpha (D)$ continuously. In fact, more is true. The following result is essential to our method and seems to have been overlooked; and we give here a complete proof using lemmas above.

\begin{thm} If $f\in C^\alpha(D)$, then $\mathcal{N}(f)\in C^{2+\alpha}(D)$. Moreover, there is a constant $C(n,\alpha)$, independent of $R$, such that
$$\|\mathcal{N}(f)\|^{(2)}\leq C(n,\alpha) \|f\|.$$
\end{thm}
\begin{proof} First let us recall the norm
$$\|\mathcal{N}(f)\|^{(2)}=\max_{1\leq i,j\leq n}\{\|\partial_{ij} \mathcal{N}(f)\|\}.$$
Let
 $$\phi(x)=\int_D \partial_{ij}\Gamma(x-y)(f(y)-f(x))dy.$$
Then, $\partial_{ij}\mathcal{N}(f)=\phi(x)-\frac{\delta_{ij}}{n}f(x)$, and therefore
$$\|\partial_{ij}\mathcal{N}(f)\|\leq \|\phi\|+\|f\|.$$
Now we proceed to estimate $\|\phi\|$. First, for $x\in \mathrm{Int} (D)$,
\begin{eqnarray}
|\phi(x)|&\leq& \int_D |\partial_{ij}\Gamma(x-y)||f(y)-f(x)|dy\nonumber\\
       &\leq & CH_\alpha[f] \int_D \frac{|y-x|^\alpha}{|x-y|^n}dy\nonumber\\
       &\leq & CH_\alpha[f]\int_{S^{n-1}}\int_0^{2R}\frac{r^\alpha r^{n-1}}{r^n}drdS=C(n,\alpha)H_\alpha[f] R^\alpha.
\end{eqnarray}
To compute the H\"older constant of $\phi$, let $x, x'$ be two (distinct) points of $D$. Let $B(x,\rho)$ be open ball of radius $\rho=2|x-x'|$ with center at $x$. Let $D_0=D\cap B(x,\rho)$. We consider
\begin{eqnarray}
\phi(x)-\phi(x')&=&\int_D \partial_{ij}\Gamma(x-y)(f(y)-f(x))dy-\int_D \partial_{ij}\Gamma(x'-y)(f(y)-f(x'))dy\nonumber\\
&=&\int_{D\setminus D_0} \partial_{ij}\Gamma(x-y)(f(y)-f(x))dy+\int_{D_0} \partial_{ij}\Gamma(x-y)(f(y)-f(x))dy\nonumber\\
&-&\int_{D\setminus D_0} \partial_{ij}\Gamma(x'-y)(f(y)-f(x'))dy-\int_{D_0} \partial_{ij}\Gamma(x'-y)(f(y)-f(x'))dy\nonumber\\
&=&\int_{D\setminus D_0} (\partial_{ij}\Gamma(x-y)-\partial_{ij}\Gamma(x'-y))(f(y)-f(x))dy\nonumber\\
&+&(f(x')-f(x))\int_{D\setminus D_0} \partial_{ij}\Gamma(x'-y)dy\nonumber\\
&+&\int_{D_0} \partial_{ij}\Gamma(x-y)(f(y)-f(x))dy-\int_{D_0} \partial_{ij}\Gamma(x'-y)(f(y)-f(x'))dy\nonumber\\
&=&I_1+I_2+I_3+I_4.
\end{eqnarray}
We are ready to estimate each of $I_j$:
\begin{eqnarray}
|I_1|&=&\bigg|\int_{D\setminus D_0} (\partial_{ij}\Gamma(x-y)-\partial_{ij}\Gamma(x'-y))(f(y)-f(x))dy\bigg|\nonumber\\
&\leq & H_\alpha[f]|x-x'|\int_{D\setminus D_0}|\nabla\partial_{ij}\Gamma(\hat{x}-y)||x-x'|^\alpha dy\nonumber\\
&\leq& C H_\alpha[f]|x-x'|\int_{D\setminus D_0}|x-x'|^\alpha |\hat{x}-y|^{-n-1}dy\nonumber\\
&\leq& C H_\alpha[f]|x-x'|\int_\rho^{2R} r^\alpha r^{-n-1} r^{n-1}dr\nonumber\\
&=& C H_\alpha[f] \rho^\alpha.
\end{eqnarray}
where $\hat{x}$ is a point on the line segment between $x, x'$, and a polar coordinate is used at $\hat{x}$, for which we have
for $y\in D\setminus D_0$, $\rho\leq |y-\hat{x}|\leq 2R$.

By Lemma 3.2, which is a crucial difference from [GT. Lemma 4.4],
we have
\begin{eqnarray}
|I_2|&=&\bigg|(f(x')-f(x))\int_{D\setminus D_0} \partial_{ij}\Gamma(x'-y)dy\bigg|\nonumber\\
&\leq & CH_\alpha[f] |x-x'|^\alpha.
\end{eqnarray}
Similarly,
\begin{eqnarray}
|I_3|&=&\bigg|\int_{D_0} \partial_{ij}\Gamma(x-y)(f(y)-f(x))dy\bigg|\nonumber\nonumber\\
&\leq & H_\alpha[f] \int_{D_0}|\partial_{ij}\Gamma(x-y)||x-y|^\alpha dy\nonumber\nonumber\\
&\leq & CH_\alpha[f] \int_{D_0}|x-y|^{-n}||x-y|^\alpha dy\nonumber\\
&=&CH_\alpha[f] |x-x'|^\alpha.
\end{eqnarray}
For $I_4$, the estimate is identical to $I_3$.
Combining (15)-(19), we complete the proof.
\end{proof}

\section{An integral system}
We consider the integral system of (4)
\begin{eqnarray}
u=h+\mathcal{N}(a).
\end{eqnarray}
Namely,
\begin{eqnarray}
u_1 &=&h_1+\mathcal{N}(a^1)\nonumber\\
u_2 &=&h_2+\mathcal{N}(a^2)\nonumber\\
&\vdots&\nonumber\\
u_N &=&h_N+\mathcal{N}(a^N)\nonumber
\end{eqnarray}
where $h=(h_1(x), ..., h_N(x))$ with $h_j(x)$ being any harmonic function, and $\mathcal{N}(a)$ is the Newtonian potential of $a$, namely
$$\mathcal{N}(a^i)(x)=\int_D \Gamma(x-y)a^i(y, u(y), \nabla u(y), \nabla^2 u(y))dy.$$
It is clear that any solution of (20) is a solution to (4). We further modify the equation (20) to fit our Banach space $C_0^{2+\alpha}(D)$. To do so, we let, for any $f\in C^{2+\alpha}(D)$, $i=1,...,N$,
\begin{eqnarray}
\omega^i(f)(x)=\mathcal{N}(a^i(y, f(y), \nabla f(y), \nabla^2 f(y))(x)
\end{eqnarray}
and define
\begin{eqnarray}
\Theta^i(f)(x)=\omega^i(f)(x)-\omega^i(f)(0)-\sum_{j=1}^n\partial_j(\omega^i(f))(0)x_j-\frac{1}{2}\sum_{k\not =l;k,l=1}^n\partial_k\partial_l(\omega^i(f))(0)x_kx_l.
\end{eqnarray}
First we remark that the last term subtracted is harmonic since $k\not=l$. Therefore $\Delta \Theta^i(f)=\Delta \omega^i(f)$.
We note that by Theorem 3.4, if $f\in C^{2+\alpha}(D)$, then $\omega^i(f)\in C^{2+\alpha}(D)$, and $\Theta^i(f)\in C_0^{2+\alpha}(D)$ by construction (22). More importantly, we have
\begin{eqnarray}
\partial_k\partial_l( \Theta^i(f))(0)=0
\end{eqnarray}
if $k\not =l $.
Now we introduce the Banach space from which we will seek our solutions. Define
$$\mathbf{B}(R)=C_0^{2+\alpha}(D)\times\cdot\cdot\cdot \times C_0^{2+\alpha}(D)$$
which  consists of $N$-copies of $C_0^{2+\alpha}(D)$, and define the function on $\mathbf{B}(R)$ as
$$\|f\|^{(2)}=\max_{1\leq j\leq N}\|f_j\|^{(2)}.$$
By Lemma 2.8, $(\mathbf{B}(R), \|\cdot\cdot\cdot\|^{(2)})$ is a Banach space.
Now we are ready to consider a map as follows
$$\Theta: \mathbf{B}(R)\to \mathbf{B}(R)$$
$$\Theta(f)=(\Theta^1(f), ..., \Theta^N(f)).$$
In order to apply a fixed point theorem on the Banach space $\mathbf{B}(R)$, which has $R$ as a parameter, we will have to estimate $\|\Theta(f)-\Theta(g)\|^{(2)}$, and $\|\Theta(f)\|^{(2)}$ for
any $f, g\in\mathbf{B}(R)$. This is to be done in the next subsections.
\subsection{ $\|\Theta(f)-\Theta(g)\|^{(2)}$ estimates}
It suffices to estimate $\|\Theta^i(f)-\Theta^i(g)\|^{(2)}$. To this end, we see from (22)
\begin{eqnarray}
\|\Theta^i(f)-\Theta^i(g)\|^{(2)}\leq \|\omega^i(f)-\omega^i(g)\|^{(2)}+\sum_{k,l=1}^n|\partial_k\partial_l(\omega^i(f)-\omega^i(g))(0)|.
\end{eqnarray}
First we have $\|\omega^i(f)-\omega^i(g)\|^{(2)}=\|\mathcal{N}(a^i(\cdot, f, \nabla f, \nabla^2 f)-a^i(\cdot, g, \nabla g, \nabla^2 g))\|^{(2)}$ by (21). By Theorem 3.4, we have $\|\mathcal{N}(\phi)\|^{(2)}\leq C\|\phi\|$ for any $\phi\in C^\alpha(D)$, where $C$ is a constant dependent only on $n$, in particular, independent of the radius $R$. Therefore we have
$$\|\omega^i(f)-\omega^i(g)\|^{(2)}\leq C\|a^i(\cdot, f, \nabla f, \nabla^2 f)-a^i(\cdot, g, \nabla g, \nabla^2 g))\|.$$

Now we are estimating $\|a^i(\cdot, f,...)-a^i(\cdot, g,...)\|$. First we use coordinates for $x=(x_k), p=(p_j), q=(q_k^j),$ and $r=(r_{kl}^j)$. From this point on, we will use constant $C$ depending only on $n,N,\alpha$, which varies from line to line. We begin with

\begin{eqnarray}
&&a^i(x, f, \nabla f, \nabla^2 f)-a^i(x, g, \nabla g, \nabla^2 g)\nonumber\\
&=& \int_0^1\frac{d}{dt}a^i(x, tf+(1-t)g, t\nabla f+(1-t)\nabla g, t\nabla^2 f+(1-t)\nabla^2 g)dt\nonumber\\
&=& \sum_{j=1}^N A_j(f_j-g_j)+\sum_{k=1}^n\sum_{j=1}^NB_k^j\partial_k(f_j-g_j)+\sum_{j=1}^N\sum_{k,l=1}^nC_{kl}^j\partial_k\partial_l(f_j-g_j)
\end{eqnarray}
where
\begin{eqnarray}
A_j&=&\int_0^1 \frac{\partial}{\partial p_j}a^i(x, tf+(1-t)g, t\nabla f+(1-t)\nabla g, t\nabla^2 f+(1-t)\nabla^2 g)dt\\
B_k^j&=&\int_0^1 \frac{\partial}{\partial q_k^j}a^i(x, tf+(1-t)g, t\nabla f+(1-t)\nabla g, t\nabla^2 f+(1-t)\nabla^2 g)dt\\
C_{kl}^j&=&\int_0^1 \frac{\partial}{\partial r_{kl}^j}a^i(x, tf+(1-t)g, t\nabla f+(1-t)\nabla g, t\nabla^2 f+(1-t)\nabla^2 g)dt.
\end{eqnarray}
Taking norm $\|\cdot\|$ on (25), and using algebraic property of the norm, we obtain
\begin{eqnarray}
&&\|a^i(x, f, \nabla f, \nabla^2 f)-a^i(x, g, \nabla g, \nabla^2 g)\|\nonumber\\
&\leq& \sum_{j=1}^N \|A_j\|\|f_j-g_j\|+\sum_{k=1}^n\sum_{j=1}^N\|B_k^j\|\|\partial_k(f_j-g_j)\|+\sum_{j=1}^N\sum_{k,l=1}^n\|C_{kl}^j\|\|\partial_k\partial_l(f_j-g_j)\|\nonumber\\
&\leq& \|f-g\|\sum_{j=1}^N \|A_j\|+\|f-g\|^{(1)}\sum_{k=1}^n\sum_{j=1}^N\|B_k^j\|+\|f-g\|^{(2)}\sum_{j=1}^N\sum_{k,l=1}^n\|C_{kl}^j\|\nonumber\\
&\leq& C(R^2\sum_{j=1}^N \|A_j\|+R\sum_{k=1}^n\sum_{j=1}^N\|B_k^j\|+\sum_{j=1}^N\sum_{k,l=1}^n\|C_{kl}^j\|)\|f-g\|^{(2)}
\end{eqnarray}
where we have used, according to Lemma 2.5, that
$$\|f-g\|\leq CR^2\|f-g\|^{(2)}, \mbox{ and } \|f-g\|^{(1)}\leq CR\|f-g\|^{(2)}.$$
Since we will apply Fixed point theorem for a closed subset of $\mathbf{B}(R)$, we consider the following closed subset
$$\mathbf{A}(R, \gamma)=\{f\in \mathbf{B}(R)|\|f\|^{(2)}\leq \gamma\}.$$
Let
$$\mathcal{W}^k=t\nabla^k f(x)+(1-t)\nabla^k g(x),$$
for $k=0, 1, 2.$
We need to study the range of $\mathcal{W}^k$ for $f, g\in \mathbf{A}(R, \gamma)$. The following is needed.
\begin{lem}
If $f, g\in \mathbf{A}(R, \gamma)$, then
$$|\mathcal{W}^k|\leq CR^{2-k}\gamma$$
for $k=0, 1, 2.$
\end{lem}
\begin{proof}
This follows from Lemma 2.4 and 2.5 since $|\mathcal{W}^k|\leq \|\nabla^k f\|+\|\nabla^k g\|$.
\end{proof}
We now consider the compact set
$$E(R, \gamma)=D\times\{p\in \re^N||p|\leq CR^2\gamma\}\times\{q\in \re^N\otimes\re^n||q|\leq CR\gamma\}\times\{r\in\mathrm{Sym}(n)\otimes\re^n||r|\leq C\gamma\}$$

In order to carry out estimates of (29), we need to introduce constants that would measure solvability of partial differential equations we consider in this paper.
\begin{eqnarray}
A[R,\gamma]&=&\max\{\bigg|\frac{\partial a^i}{\partial p_j}\bigg|_{E(R,\gamma)}|i=1,..., N; j=1,...,n\}\\
H_\alpha^A[R,\gamma]&=&\max\{H_\alpha\bigg[\frac{\partial a^i}{\partial p_j}\bigg]_{E(R,\gamma)}|i=1,..., N; j=1,...,n\}\\
B[R,\gamma]&=&\max\{\bigg|\frac{\partial a^i}{\partial q_k^j}\bigg|_{E(R,\gamma)}|i=1,..., N; k,j=1,...,n\}\\
H_\alpha^B[R,\gamma]&=&\max\{H_\alpha\bigg[\frac{\partial a^i}{\partial q_k^j}\bigg]_{E(R,\gamma)}|i=1,..., N; k,j=1,...,n\}\\
C[R,\gamma]&=&\max\{\bigg|\frac{\partial a^i}{\partial r_{kl}^j}\bigg|_{E(R,\gamma)}|i=1,..., N; j=1,...,n\}\\
H_\alpha^C[R,\gamma]&=&\max\{H_\alpha\bigg[\frac{\partial a^i}{\partial r_{kl}^j}\bigg]_{E(R,\gamma)}|i=1,..., N; j=1,...,n\}
\end{eqnarray}
Now we need more important constants on Lipschitz in variables of $r$. Here we denote the Lipschitz constant in $r$ as follows
$$ H_1[f]=\sup\{\frac{|f(\cdot, r)-f(\cdot, r'\cdot)|}{|r-r'|}\}.$$
Now we define:
\begin{eqnarray}
H_1^A[R,\gamma]&=&\max\{H_1\bigg[\frac{\partial a^i}{\partial p_j}\bigg]_{E(R,\gamma)}|i=1,..., N; j=1,...,n\}\\
H_1^B[R,\gamma]&=&\max\{H_1\bigg[\frac{\partial a^i}{\partial q_k^j}\bigg]_{E(R,\gamma)}|i=1,..., N; k,j=1,...,n\}\\
H_1^C[R,\gamma]&=&\max\{H_1\bigg[\frac{\partial a^i}{\partial r_{kl}^j}\bigg]_{E(R,\gamma)}|i=1,..., N; j=1,...,n\}
\end{eqnarray}
\subsubsection{Estimate of $\|A_j\|$}
It is obvious that
$$|A_j|\leq A[R,\gamma].$$
Now we are estimating $H^A_\alpha[A_j]$. To do so, we begin with
\begin{eqnarray}
A_j(x)-A_j(x')&=&
\int_0^1\frac{\partial a^i}{\partial p_j}(x, \mathcal{W}^0(x),\mathcal{W}^1(x), \mathcal{W}^2(x))dt\nonumber\\
&-&\int_0^1\frac{\partial a^i}{\partial p_j}(x', \mathcal{W}^0(x'),\mathcal{W}^1(x'), \mathcal{W}^2(x'))dt\nonumber\\
&=&\int_0^1\frac{\partial a^i}{\partial p_j}(x, \mathcal{W}^0(x),\mathcal{W}^1(x), \mathcal{W}^2(x))dt\nonumber\\
&-&\int_0^1\frac{\partial a^i}{\partial p_j}(x', \mathcal{W}^0(x),\mathcal{W}^1(x), \mathcal{W}^2(x))dt\nonumber\\
&+&\int_0^1\frac{\partial a^i}{\partial p_j}(x', \mathcal{W}^0(x'),\mathcal{W}^1(x), \mathcal{W}^2(x))dt\nonumber\\
&-&\int_0^1\frac{\partial a^i}{\partial p_j}(x', \mathcal{W}^0(x'),\mathcal{W}^1(x'), \mathcal{W}^2(x))dt\nonumber\\
&+&\int_0^1\frac{\partial a^i}{\partial p_j}(x', \mathcal{W}^0(x'),\mathcal{W}^1(x'), \mathcal{W}^2(x))dt\nonumber\\
&-&\int_0^1\frac{\partial a^i}{\partial p_j}(x', \mathcal{W}^0(x'),\mathcal{W}^1(x'), \mathcal{W}^2(x'))dt
\end{eqnarray}
and it follows
\begin{eqnarray}
&&|A_j(x)-A_j(x')|\nonumber\\
&\leq & H_\alpha^A[R,\gamma]|x-x'|^\alpha+H_\alpha^A[R,\gamma]\sum_{j=1}^N(|f^j(x)-f^j(x')|+|g^j(x)-g^j(x')|)^\alpha\nonumber\\
&+&H_\alpha^A[R,\gamma]\sum_{j=1}^N(|\nabla f^j(x)-Df^j(x')|+|\nabla g^j(x)-Dg^j(x')|)^\alpha\nonumber\\
&+&H_1^A[R,\gamma]\sum_{j=1}^N(|\nabla^2f^j(x)-\nabla^2 f^j(x')|+|\nabla^2 g^j(x)-\nabla^2 g^j(x')|).
\end{eqnarray}

We remark here that $C^2$ regularity is used to Lipschitz estimate of $r$ variables, and if $a$ is independent of $r$, the regularity of $C^{1+\alpha}$ is enough for the estimate. This fact will be used in proving Theorem 1.2 and Corollary 1.4.

Now we need a lemma on Lipschitz property of $C^{m+\alpha}_0(D)$.
\begin{lem} Let $f\in C^{k+\alpha}_0(D)$.
For $x,x'\in D$, and $|\beta|\leq k-1$, we have
$$|\partial^\beta f(x)-\partial^\beta f(x')| \leq \|f\|^{(|\beta|+1)}|x-x'|.$$
\end{lem}
\begin{proof}
We have
\begin{eqnarray}
\partial^\beta f(x)-\partial^\beta f(x')&=&\int_0^1\frac{d}{dt}\partial^\beta f(tx+(1-t)x')dt\nonumber\nonumber\\
&=&\int_0^1\nabla\partial^\beta f(tx+(1-t)x')\cdot(x-x')dt.
\end{eqnarray}
It follows from (41) that
\begin{eqnarray}
|\partial^\beta f(x)-\partial^\beta f(x')|&\leq& \|f\|^{(|\beta|+1)}|x-x'|.\nonumber
\end{eqnarray}
\end{proof}
Applying Lemma 4.2 and 2.5, we have
$$|f(x)-f(x')|\leq \|f\|^{(1)}|x-x'|\leq 3n R\|f\|^{(2)}|x-x'|,$$
$$|\partial_j f(x)-\partial_j f(x')|\leq \|f\|^{(2)}|x-x'|.$$
It follows from (40)
\begin{eqnarray}
&&|A_j(x)-A_j(x')|\nonumber\\
&\leq & H_\alpha^A[R,\gamma]|x-x'|^\alpha+H^A_\alpha[R,\gamma]N(3nR)^\alpha (\|f\|^{(2)}+\|g\|^{(2)})^\alpha|x-x'|^\alpha\nonumber\\
&+&H_\alpha^A[R,\gamma]N(\|f\|^{(2)}+\|g\|^{(2)})|x-x'|^\alpha\nonumber\\
&+&H_1^A[R,\gamma]\sum_{j=1}^N(H_\alpha[\nabla^2 f^j]+H_\alpha[\nabla^2 g^j])|x-x'|^\alpha.
\end{eqnarray}
Now we have by (42)
$$|A_j(x)-A_j(x')|\leq (H_\alpha^A[R,\gamma]+2N(3nR)^\alpha\gamma^\alpha H_\alpha^A[R,\gamma] +2N\gamma H_\alpha^A[R,\gamma]+2N\gamma H_1^A[R,\gamma](2R)^{-\alpha})|x-x'|^\alpha$$
which implies
\begin{eqnarray}
H_\alpha [A_j]\leq H_\alpha^A[R,\gamma]+2N(3nR)^\alpha\gamma^\alpha H_\alpha^A[R,\gamma] +2N\gamma H_\alpha^A[R,\gamma]+2N\gamma H_1^A[R,\gamma](2R)^{-\alpha}
\end{eqnarray}
and
\begin{eqnarray}
\|A_j\|&=&|A_j|+(2R)^\alpha H_\alpha[A_j]\nonumber\\
&\leq& A[R,\gamma]+(2R)^\alpha(1+2N(3nR)^\alpha\gamma^\alpha +2N\gamma )H_\alpha^A[R,\gamma]
+2N\gamma H_1^A[R,\gamma].
\end{eqnarray}
Similarly, we have estimates
\begin{eqnarray}
\|B_k^j\|&=&|B_j|+(2R)^\alpha H_\alpha[B_k^j]\nonumber\\
&\leq& B[R,\gamma]+(2R)^\alpha(1+2N(3nR)^\alpha\gamma^\alpha +2N\gamma )H_\alpha^B[R,\gamma]
+2N\gamma H_1^B[R,\gamma],
\end{eqnarray}
and
\begin{eqnarray}
\|C_{kl}^j\|&=&|C_{kl}^j|+(2R)^\alpha H_\alpha[C_{kl}^j]\nonumber\\
&\leq& C[R,\gamma]+(2R)^\alpha(1+2N(3nR)^\alpha\gamma^\alpha +2N\gamma )H_\alpha^C[R,\gamma]
+2N\gamma H_1^C[R,\gamma],
\end{eqnarray}
To simplify the notation, we denote the right side of (44), (45), and (46) respectively by $\delta_A(R,\gamma), \delta_B(R,\gamma)$ and $\delta_C(R,\gamma)$.
Then by (29) we, have
\begin{eqnarray}
&&\|a^i(x,f,\nabla f, \nabla^2 f)-a^i(x,g,\nabla g, \nabla^2 g)\|\nonumber\\
&\leq& C(n,N,\alpha)(R^2\delta_A(R,\gamma)+R\delta_B(R,\gamma)+\delta_C(R,\gamma))\|f-g\|^{(2)}.
\end{eqnarray}
\begin{lem} Let $f\in C^\alpha(D)$. Then it holds
$$|\partial_{kl} \mathcal{N}(f)(0)|\leq C(n,\alpha)\|f\|.$$
\end{lem}
\begin{proof}
By Theorem 3.4, we have
$$|\partial_{kl}l \mathcal{N}(f)(0)|\leq |\partial_{kl}\mathcal{N}(f)|\leq \|\mathcal{N}(f)\|^{(2)}\leq C(n,\alpha)\|f\|.$$
\end{proof}

Finally, we combine (24),and Lemma 4.3
to conclude that
\begin{eqnarray}
\|\Theta^i(f)-\Theta^i(g)\|^{(2)}\leq \delta(R,\gamma)\|f-g\|^{(2)}
\end{eqnarray}
where
\begin{eqnarray}
\delta(R,\gamma)&=&C(n,N,\alpha)(R^2\delta_A(R,\gamma)+R\delta_B(R,\gamma)+\delta_C(R,\gamma))\\
\delta_A(R,\gamma)&=& A[R,\gamma]+(2R)^\alpha(1+2N(3nR)^\alpha\gamma^\alpha +2N\gamma )H_\alpha^A[R,\gamma]
+2N\gamma H_1^A[R,\gamma]\\
\delta_B(R,\gamma)&=& B[R,\gamma]+(2R)^\alpha(1+2N(3nR)^\alpha\gamma^\alpha +2N\gamma )H_\alpha^B[R,\gamma]
+2N\gamma H_1^B[R,\gamma]\\
\delta_C(R,\gamma)&=& C[R,\gamma]+(2R)^\alpha(1+2N(3nR)^\alpha\gamma^\alpha +2N\gamma )H_\alpha^C[R,\gamma]
+2N\gamma H_1^C[R,\gamma].
\end{eqnarray}
\subsection{ $\|\Theta(f)\|^{(2)}$ estimates}
Our proof here is mostly similar to that of (48). We will point out minor differences without repeating the arguments. It suffices to estimate $\|\Theta^i(f)\|^{(2)}$. To this end, we see from (22)
\begin{eqnarray}
\|\Theta^i(f)\|^{(2)}\leq \|\omega^i(f)\|^{(2)}+\sum_{k,l=1}^n|\partial_k\partial_l(\omega^i(f))(0)|.
\end{eqnarray}
First we have $\|\omega^i(f)\|^{(2)}=\|\mathcal{N}(a^i)\|^{(2)}$. By Theorem 3.4, we have $\|\mathcal{N}(a^i)\|^{(2)}\leq C\|a^i\|$, where $C$ is a constant dependent on only on $n$, in particular, independent of the radius $R$.

Now we are estimating $\|a^i\|$. First we use coordinates for $x=(x_j), p=(p_j), q=(q_k^j),$ and $r=(r_{kl}^j)$.
We begin with
\begin{eqnarray}
&&a^i(x,f,\nabla f(x), \nabla^2 f(x))-a^i(0,...,0)\nonumber\\
&=&\int_0^1 \frac{d}{dt}a^i(tx,tf,t\nabla f(x), t\nabla^2 f(x))dt\nonumber\\
&=& \sum_{j=1}^n D_j x_j+\sum_{j=1}^N A_j(f_j)+\sum_{k=1}^n\sum_{j=1}^NB_k^j\partial_k(f_j)+\sum_{j=1}^N\sum_{k,l=1}^nC_{kl}^j\partial_k\partial_l(f_j)
\end{eqnarray}
where
\begin{eqnarray}
A_j&=&\int_0^1 \frac{\partial}{\partial p_j}a^i(x, t f, t\nabla f, t\nabla^2 f)dt\\
B_k^j&=&\int_0^1 \frac{\partial}{\partial q_k^j}a^i(x, t f, t\nabla f, t\nabla^2 f)dt\\
C_{kl}^j&=&\int_0^1 \frac{\partial}{\partial r_{kl}^j}a^i(x, t f, t\nabla f, t\nabla^2 f)dt\\
D_j &=&\int_0^1 \frac{\partial}{\partial x_j}a^i(x, t f, t\nabla f, t\nabla^2 f)dt.
\end{eqnarray}
We notice here we have kept notations similar to (26)-(28), although the integrand is different.
By virtue of the same argument as for (24), we can arrive to the following estimate
\begin{eqnarray}
&&\|a^i(x, f, \nabla f, \nabla^2 f)\|\leq |a^i(0)|+3Rn\|D_j\|\nonumber\\
&+& C(R^2\sum_{j=1}^N \|A_j\|+R\sum_{k=1}^n\sum_{j=1}^N\|B_k^j\|+\sum_{j=1}^N\sum_{k,l=1}^n\|C_{kl}^j\|)\|f\|^{(2)}.
\end{eqnarray}

Similarly, we need to define constants $D[R,\gamma)], H_\alpha^D[R,\gamma],$ and $ H_1^D[R,\gamma]$ as in (30)-(38). Without repeating, we will come out estimate
\begin{eqnarray}
\|\Theta^i(f)\|^{(2)}\leq \eta(R,\gamma)
\end{eqnarray}
where
\begin{eqnarray}
\eta(R,\gamma)&=&C(n,N,\alpha)(|a(0)|+R\delta_D(R,\gamma)+\gamma\delta(R,\gamma))\\
\delta_D(R,\gamma)&=&D[R,\gamma]+(2R)^\alpha(1+2N(3nR)^\alpha\gamma^\alpha +2N\gamma )H_\alpha^D[R,\gamma]
+2N\gamma H_1^D[R,\gamma],
\end{eqnarray}
and $\delta (R,\gamma)$ is given by (49).
\subsection{A general estimate}
Here we collect the estimates together for a later quick reference.
\begin{thm}
Let $\Theta: \mathbf{B}(R)\to\mathbf{B}(R)$ be defined as in (22). If $f, g\in \mathbf{A}(R,\gamma)$, then
\begin{eqnarray}
\|\Theta(f)-\Theta(g)\|^{(2)}&\leq& \delta(R,\gamma)\|f-g\|^{(2)}\\
\|\Theta(f)\|^{(2)}&\leq& \eta(R,\gamma)
\end{eqnarray}
where $\delta(R,\gamma)$ and $\eta(R,\gamma)$ are defined by (49), (61) respectively.
\end{thm}
\section{Proof of theorems for Possion type}
In this section, we will give proofs of all results presented in the introduction. It suffices to work with $a$ of class $C^2$ or $C^{1+\alpha}$ by the regularity theory of Laplace.
\subsection{Non-radial functions}
In order to construct solutions that are not radial, it is helpful to have the following simple lemma for this purpose.

\begin{lem}Let $u$ be any function of $C^2(D)$. If $\nabla^2 u(0)$ is not $\lambda I_{n\times n}$ for some $\lambda\in \re$, then $u$ is not radial.
\end{lem}
\begin{proof} If $u$ is radial, then there is a function $v$ such that $u(x)=v(r)$ where $x=r\omega, \omega\in S^{n-1}$. Obviously $v$ is $C^2$ for
$r>0$. In fact, we have, for $r>0$,
$$v''(r)=\sum_{k,l=1}^n\frac{\partial^2 u}{\partial x_k\partial x_l} (x)\omega_k\omega_l$$
Taking limit on both side of the equation above, we have
\begin{eqnarray}
\lim_{r\to 0}v''(r)&=&\sum_{k,l=1}^n\lim_{x\to 0}\frac{\partial^2 u}{\partial x_k\partial x_l} (x)\omega_k\omega_l\nonumber\\
&=&\sum_{k,l=1}^n\frac{\partial^2 u}{\partial x_k\partial x_l} (0)\omega_k\omega_l\nonumber.
\end{eqnarray}
In particular $\lim_{r\to 0} v''(r)$ exists and we denote it by $\lambda$, and $\nabla^2 u(0)$ by $A$.
We conclude that $\lambda=\omega A\omega^{\perp}$ for any $\omega\in S^{n-1}$; namely, $\omega (A-\lambda I)\omega^{\perp}=0$ for any $\omega\in S^{n-1}$. This means that $A=\lambda I$, a contradiction.
\end{proof}
\subsection{Proof of Theorem 1.1}
Here we will prove a result slightly more general than Theorem 1.1.
\begin{thm} Let $a(x, p, q, r)=(a_1(x, p, q, r), ..., a_N(x, p, q, r))$ be of class $C_{loc}^{k}$ ($2\leq k\leq\infty, 0<\alpha<1)$,  where $x\in\re^n, p\in\re^N, q\in \re^n\otimes\re^N$, and $r\in \mathrm{Sym}(n)\otimes\re^N$. There is a (small) constant $\delta$ depending on $n,N,\alpha$ such that if
\begin{eqnarray}
a(0)&=& 0,\\
|\nabla_r a(0)|+|\nabla_r^2 a(0)|&<&\delta,
\end{eqnarray}
then the following system: $u(x)=(u_1(x), ..., u_N(x)): \{|x|\leq R\}\to \re^N$,
\begin{eqnarray}
\Delta u(x)&=&a(x, u(x), \nabla u(x), \nabla^2 u(x))
\end{eqnarray}
has infinitely many solutions of $C^{k+2+\alpha}(D)$ of vanishing order two at the origin for sufficiently small values of $R$. Furthermore these solutions are not radially symmetric.
\end{thm}
In the proof, our goal is to find $R, \gamma$ sufficiently small so that we have
\begin{eqnarray}
\delta(R,\gamma)&<&1,\nonumber\\
\eta(R,\gamma)&<&\frac{\gamma}{2}.\nonumber
\end{eqnarray}
Once we have these, we can consider an operator $\mathcal{T}$ defined by
$$\mathcal{T}(f)=h+\Theta(f)$$
where $f\in\mathbf{A}(R,\gamma)$ and $h$ is a harmonic vector-function such that $h(0)=\nabla h(0)=0$, and $\|h\|^{(2)}=\frac{\gamma}{2}$.
Therefore $\mathcal{T}$ is a contration operator from $\mathbf{A}(R,\gamma)$ to $\mathbf{A}(R,\gamma)$, for which there is a fixed point that becomes a
solution of (20).

To this end, we first simply $\delta(R,\gamma),\eta(R,\gamma)$, and we can replace them by
\begin{eqnarray}
\delta(R,\gamma)&=&C(C[R,\gamma]+\gamma H_1^C[R,\gamma])+\varepsilon(R,\gamma)\\
\eta(R,\gamma)&=&C(|a(0)|+\gamma C[R,\gamma]+\gamma^2H_1^C[R,\gamma])+\varepsilon(R,\gamma).
\end{eqnarray}
where $C$ is a constant only depending on $n,N,\alpha$, and $\varepsilon(R,\gamma)$ is such that $\lim_{R\to 0}\varepsilon(R,\gamma)=0$ for each $\gamma>0$. We now give estimates of $C[R,\gamma], H_1^C[R,\gamma]$ in terms of conditions (1)-(3). Let $\sigma=r_{kl}^j$ be a component variable of $r$. We want to estimate the Lipschitz constant of $\partial_\sigma a$. In fact, we have
\begin{eqnarray}
&&\partial_\sigma a^i(x, p, q, r)-\partial_\sigma a^i(x, p, q, r')\nonumber\\
&=& \int_0^1\frac{d}{dt}a^i(x, p, q, tr+(1-t)r')dt\nonumber\\
&=&\int_0^1 \sum\partial_{r_{kl}^j}\partial_\sigma a^i(\cdot)(r_{kl}^j-r'{_{kl}^j})dt.
\end{eqnarray}
Here $a^i(\cdot)$ is defined naturally as given in the integrand. It follows from (70) that
\begin{eqnarray}
H_1[\partial_\sigma a^i]|_{E(R,\gamma)}&\leq & |\nabla_r^2 a^i|_{E(R,\gamma)}\nonumber\\
&= &|\nabla_r^2 a^i(0)|+o(R+\gamma),
\end{eqnarray}
where $o(R+\gamma)\to 0$ as $R, \gamma\to 0$ by continuity of $C^2$ smoothness of $a$ and $\lim_{R,\gamma\to 0}E(R,\gamma)=\{0\}$. From (71) and definition (38), we have
$$H_1^C[R,\gamma]\leq |\nabla_r^2 a(0)|+o(R+\gamma).$$
On the other hand, we have
\begin{eqnarray}
&&\partial_\sigma a^i(x, p, q, r)-\partial_\sigma a^i(0)\nonumber\\
&=& \int_0^1\frac{d}{dt}a^i(tx, tp, tq, tr)dt\nonumber\\
&=&\int_0^1 \bigg\{\sum \partial_{x_j}\partial_\sigma a^i(\cdot) x_j+\sum \partial_{p_j}\partial_\sigma a^i(\cdot)p_j\nonumber\\
&+&\sum \partial_{q_k^j}\partial_\sigma a^i(\cdot)q_k^j+ \sum\partial_{r_{kl}^j}\partial_\sigma a^i(\cdot)r_{kl}^j\bigg\}dt.
\end{eqnarray}
Notice that for $(x,p, q, r)\in E(R,\gamma)$, we have $|x|\leq R, |p|\leq CR^2\gamma, |q|\leq CR\gamma$, and $|r|\leq C\gamma$. Hence by (72), putting terms with $R, \gamma$ together, we have
$$|\partial_\sigma a^i|_{E(R,\gamma)}\leq C(|\nabla_r a^i(0)|+|\nabla_r^2 a^i(0)|)+\varepsilon(R,\gamma)$$
where $\varepsilon(R,\gamma)\to 0$ as $R\to 0$ for each given $\gamma$. This implies
$$C[R,\gamma]\leq C(|\nabla_r a(0)|+|\nabla_r^2 a(0)|)+\varepsilon(R,\gamma)$$
Letting
$$\tau=|\nabla_r a(0)|+|\nabla_r^2 a(0)|,$$
we can have
$$\delta(R,\gamma)=C(\tau+\gamma(\tau+o(R+\gamma))+\varepsilon(R,\gamma)$$
$$\eta(R,\gamma)=C(\gamma\tau)+\gamma^2(\tau+o(R+\gamma))+\varepsilon(R,\gamma).$$
By $o(R+\gamma)\to 0$, there exist $R_0,\gamma_0 (<1)$ such that
$$o(R+\gamma)\leq C\tau$$
for $R\leq R_0, \gamma\leq \gamma_0$.
Hence  we have
$$\delta(R,\gamma_0)\leq C\tau+\gamma_02C\tau+\varepsilon(R,\gamma_0)$$
$$\eta(R,\gamma_0)\leq \gamma_03C\tau+\varepsilon(R,\gamma_0).$$
Now choose $\delta=\frac{3}{10C}$. If $\tau<\delta$, we have
$$\delta(R,\gamma_0)\leq \frac{3}{10}+\frac{2\gamma_0}{10}+\varepsilon(R,\gamma_0)$$
$$\eta(R,\gamma_0)\leq \frac{3\gamma_0}{10}+\varepsilon(R,\gamma_0).$$
Finally, we choose $R$ small so that $\varepsilon(R,\gamma_0)\leq \min\{ \frac{4}{10},\frac{2\gamma_0}{10}\}$. It follows that
$$\delta(R,\gamma_0)\leq \frac{9}{10}$$
$$\eta(R,\gamma_0)\leq \frac{\gamma_0}{2}.$$
This is equivalent to that for $f, g\in \mathbf{A}(R,\gamma_0)$
\begin{eqnarray}
\|\Theta(f)-\Theta(g)\|^{(2)}&\leq& \frac{9}{10}\|f-g\|^{(2)}\\
\|\Theta(f)\|^{(2)}&\leq& \frac{\gamma_0}{2}.
\end{eqnarray}
Now we are ready to apply Fixed point theorem for $R,\gamma_0$. Let $h$ be any harmonic homogenous polynomial of degree $2$  so that $\|h\|^{(2)}=
\frac{\gamma_0}{2}$. Here let us be more specific. Let $h=\sum_{k,l=1}^n a_{kl}x_kx_l$ where $a_{kl}\in\re^N, a_{kl}=a_{lk}$. It is easy to see that $h$ is harmonic if and only if the trace of $a$ is zero, i.e., $\sum_{k=1}^n a_{kk}=\{0\}$. We also see that
$$\|h\|^{(2)}=\max_{1\leq k,l\leq n}|a_{kl}|.$$
So we take a harmonic polynomial of degree $2$ such that  $0<\max_{1\leq k,l\leq n}|a_{kl}|<\frac{\gamma_0}{2}$.
Then we
consider the map
$$\mathcal{T}(u)=h+\Theta(u)$$
which maps $\mathbf{A}(R,\gamma_0)$ to $\mathbf{A}(R,\gamma_0)$ as contraction map by (73),(74).
So $\mathcal{T}$ has a fixed point $u$ in $\mathbf{A}(R,\gamma_0)$. We claim the vanishing order of $u$ at the origin is $2$. In fact, by (23),
 $$\partial_k\partial_l u(0)=\partial_k\partial_l h(0)=a_{kl}
\not=\{0\}$$
 for $k\not=l$,
and also $u(0)=0$ and $\partial_k u(0)=0$ by the construction. At the same time, for different $\{a_{kl}\}$ we have different solutions $u$. Now we prove all these solutions are not radial. Assume there is $i_0$ such that
$(a^{i_0}_{kl})\not=\{0\}$. If $(a^{i_0}_{kl})=\lambda I$ for some $\lambda\in\re$, then since the trace of $(a^{i_0}_{kl})$ is zero, we conclude that $\lambda=0$, which implies $(a^{i_0}_{kl})=0$, a contradiction. By Lemma 5.1, $u^{i_0}$ is not radial. So $u$ is not radial neither.
 This completes the proof.
\begin{rem} Given $\gamma_0$ as in the proof above, we actually prove that the solution space can be parameterized by at least $\frac{n^2-n}{2}N$ parameters from the coefficients of the harmonic polynomials of degree $2$. In fact, the different choice of $a_{kl}$ produces different $u$ which is defined in the same domain of radius $R$ and with different Hessian at the origin since $\partial_k\partial_l u(0)=\partial_k\partial_l h(0)=a_{kl}$ by (23). Therefore there exist infinitely many solutions of vanishing order two at the origin. This remark also applies to other theorems
\end{rem}

\subsection{Proof of Theorem 1.2}
First we consider the case where $c_0=0, c_1=0$ as in (6), (7). Since $a$ is independent of $r$, we have $C[R,\gamma]=H_1^C[R,\gamma]=H_\alpha^C[R,\gamma]=0$. Here we note that $C^{1,\alpha}$ regularity of $a$ is only needed with a careful inspection of the proof of Theorem 4.4. Therefore we can replace constants as follows
$$\delta(R,\gamma)=\varepsilon(R,\gamma)$$
$$\eta(R,\gamma)=C|a(0)|+\epsilon(R,\gamma)$$
where $$\lim_{R\to 0}\varepsilon(R,\gamma)=\lim_{R\to 0}\epsilon(R,\gamma)=0$$ for each $\gamma>0$.
Now we choose $\gamma_0$ large enough that $\frac{\gamma_0}{4}>C|a(0)|$. Then we choose $R$ sufficiently small that $\varepsilon(R,\gamma_0)<\frac{1}{2}$ and $\epsilon(R,\gamma_0)<\frac{\gamma_0}{4}$. As a result, we have
$$\delta(R,\gamma_0)<\frac{1}{2}$$
$$\eta(R,\gamma_0)<\frac{\gamma_0}{2}.$$
As in Theorem 1.1, we find a solution in $\mathbf{A}(R,\gamma_0)$ which vanishes up to order $1$ at the origin. To get general case,
we consider a new system

$$\Delta \tilde{u}=a(x, \tilde{u}+c_0+c_1\cdot x, \nabla (\tilde{u}+c_0+c_1\cdot x))=\tilde{a}(x,\tilde{u},\nabla \tilde{u})$$
We can solve this system for $\tilde{u}$. Then $u=\tilde{u}+c_0+c_1\cdot x$ is the solution we are seeking for.

\subsection{Proof of Theorem 1.3}
Here we fix $R$, and choose $\gamma$ small to prove the existence of semi-global solutions. Since $a$ is independent of $x$, so we have
$D[R,\gamma]=H_\alpha^D[R,\gamma]=H_1^D[R,\gamma]=0$. Therefore  we can replace $\delta(R,\gamma), \eta(R,\gamma)$, using (8), (61), by
\begin{eqnarray}
\eta(R,\gamma)=C\gamma\delta(R,\gamma)
\end{eqnarray}
where $\delta(R,\gamma)$ is still given by (49). It suffices to prove that
for each given $R$, we have
\begin{eqnarray}
\lim_{\gamma\to 0}\delta(R,\gamma)=0.
\end{eqnarray}
If this is proved, then we can choose $\gamma_0$ so that
$$\delta(R,\gamma_0)<\frac{1}{2}$$
$$\eta(R,\gamma_0)<\frac{\gamma_0}{2}.$$
Our next goal is to show (76). Indeed, since $a$ is independent of $x$, we will take $E(R,\gamma)$ as
$$E(R, \gamma)=\{p\in \re^N||p|\leq CR^2\gamma\}\times\{q\in \re^N\otimes\re^n||q|\leq CR\gamma\}\times\{r\in\mathrm{Sym}(n)\otimes\re^n||r|\leq C\gamma\}.$$
We notice that as set,
$$\lim_{\gamma\to 0}E(R,\gamma)=\{0\}.$$
This is important for what follows in proving (76). Let $\sigma$ be one of component variables of $\{p, q, r\}$. We have by condition (9)
\begin{eqnarray}
\partial_\sigma a^i(p,q,r)&=&\int_0^1\frac{d}{dt}a^i(tp,tq,tr)dt\nonumber\\
&=&\int_0^1\sum \partial_{p_j}\partial_\sigma a^i(\cdot)p_j+\sum\partial_{q_k^l}\partial_\sigma a^i(\cdot)q_k^l+\sum\partial_{r_{kl}^j}\partial_\sigma a^i(\cdot)r_{kl}^j dt.
\end{eqnarray}
Since $|p_j|\leq CR^2\gamma, |q_k^j|\leq CR\gamma$, and $|r_{kl}^j|\leq C\gamma$, we have from (77)
$$|\partial_\sigma a^i|_{E(R,\gamma)}\leq C\|a\|_{C^2(E(R,\gamma))}\gamma,$$
which implies, by definitions (30),(32),(34),
\begin{eqnarray}
A[R,\gamma],B[R,\gamma],C[R,\gamma]\leq C\|a\|_{C^2(E(R,\gamma))}\gamma.
\end{eqnarray}
Here, of course, $C^2(E(R,\gamma))$ denotes the maximum norm of $\nabla^2 a$ on the set $E(R,\gamma)$.
On the other hand, we have
\begin{eqnarray}
&&\partial_\sigma a^i(p,q,r)-\partial_\sigma a^i(p',q',r')\nonumber\\
&=&\int_0^1\frac{d}{dt}a^i(tp+(1-t)p',tq+(1-t)q',tr+(1-t)r')dt\nonumber\\
&=&\int_0^1\sum \partial_{p_j}\partial_\sigma a^i(\cdot)(p_j-p_j')+\sum\partial_{q_k^l}\partial_\sigma a^i(\cdot)(q_k^l-{q'}_k^j)+\sum\partial_{r_{kl}^j}\partial_\sigma a^i(\cdot)(r_{kl}^j -{r'}_{kl}^j)dt.
\end{eqnarray}
From (79), we have, by definition (36), (37), (38),
\begin{eqnarray}
H_\alpha^A[R,\gamma],H_\alpha^B[R,\gamma],H_\alpha^C[R,\gamma]\leq C\|a\|_{C^2(E(R,\gamma))}\gamma^{1-\alpha}.
\end{eqnarray}
Similarly, from (79), we have
\begin{eqnarray}
H_1^A[R,\gamma],H_1^B[R,\gamma],H_1^C[R,\gamma]\leq C\|a\|_{C^2(E(R,\gamma))}.
\end{eqnarray}
Finally, we substitute (78), (80), and (81) into (50),(51), (52) and (49), we see that $\delta(R,\gamma)$ is a function in $\gamma$ and $\gamma^{1-\alpha}$, which proves (76). The rest of proof is similar and we omit it.

\subsection{Proof of Corollary 1.4}
It follows from Theorem 1.3 immediately since $C^{1+\alpha}$ is only needed due to the remark in the proof of Theorem 1.1.

\section{Proof of Theorem A, B, C}
\subsection{Constant coefficients}

We first point out that the results are easy consequences of Poisson type if the coefficients of $L$ are constant using a linear transformation ([GT]).
Indeed, let $\mathbf{P}$ be a constant matrix which defines a nonsingular linear transformation $y=x\mathbf{P}$ from $\re^n$ to $\re^n$. Letting
$\widetilde{u}(y)=u(x)=u(y\mathbf{P}^{-1})$, one can verify that
$$\sum a^{ij} D_{ij}u(x)=\sum \widetilde{a}^{ij} D_{ij} \widetilde{u}(y)$$
where $\mathbf{A}=[a^{ij}]$ and $\widetilde{\mathbf{A}}=\mathbf{P}^t\mathbf{A}\mathbf{P}$.
By ellipticity, we can choose $\mathbf{P}$ so that $\widetilde{\mathbf{A}}$ is identity, Then Theorem 1.1-1.3 apply to the constant coefficient case.
\subsection{Proof of Thereom A}
In order to prove this result, we need to extend Theorem 1.2 as follows
\begin{thm} Let $a(x, p, q)=(a_1(x, p, q), ..., a_N(x, p, q))$ and $b(x)=(b_1(x),...,b_N(x))$ be of class $C_{loc}^{k+\alpha}$ ($1\leq k\leq\infty, 0<\alpha<1)$,  where $x\in\re^n, p\in\re^N$, and $q\in \re^n\otimes\re^N$. Assume $b(0)=0$.
Then, for any given $c_0\in \re^N, c_1\in \re^n\otimes\re^N $, the following system: $u(x)=(u_1(x), ..., u_N(x)): \{|x|\leq R\}\to \re^N$,
\begin{eqnarray}
\Delta u(x)&=&a(x, u(x), \nabla u(x))+b(x)\cdot\nabla^2 u(x)\nonumber\\
u(0)&=&c_0\nonumber\\
\nabla u(0)&=&c_1\nonumber
\end{eqnarray}
has infinitely many solutions of $C^{k+2+\alpha}(\{|x|\leq R\})$ for sufficiently small values of $R$. In particular, all hese solutions are not radially symmetric.
\end{thm}
\begin{proof}
Consider the new sytem $\Delta u=\tilde{a}(x, u, \nabla u, \nabla^2 u)$, where 
$$\tilde{a}(x, u, \nabla u, \nabla^2 u)=a(x, u, \nabla u)+b(x)\cdot\nabla^2 u(x).$$
We note that $\tilde{a}(x, p, q, r)$ is linear in $r$ and $b(0)=0, \tilde{a}(0)=a(0)$. So the argument of Theorem 1.2 can be easily modified to give a proof. For example, we can easily check that $H_\alpha ^C[R,\gamma]=O(R^{1-\alpha}), H_1^C[R,\gamma]=0, $ and $ C[R,\gamma]=O(R^\alpha \gamma)$. We omit the rest of details.
\end{proof}

Now we are ready to give a proof of Theorem A. Let us consider the following
$$\sum_{ij} a^{ij}(0)D_{ij} u(x)=a(x, u, \nabla u)+\sum_{ij} (a^{ij}(0)-a^{ij}(x))\cdot D_{ij} u(x).$$
This is equivalent to $L u=a(x, u, \nabla u)$ and so Theorem 6.1 for constant coefficient elliptic operator applies. This completes the proof.

By the same reasoning, Theorem B, C can be similarly proved.
\subsection{Proof of Theorem D}
Let us first recall the basic definition of harmonic maps. Assume that $M, N$ have dimenions $m,n$ repsectively. If we use local coordinates, the metric tensor of $M$ can be written as 
$$(\gamma_{\alpha\beta})_{\alpha,\beta=1,...m,}$$
and the one of $N$ as 
$$(g_{ij})_{i,j=1,...n.}$$
We shall also use the following notations
$$(\gamma^{\alpha\beta}_{\alpha,\beta=1,...m}=(\gamma_{\alpha\beta})^{-1}_{\alpha,\beta=1,...m,}$$
$$\gamma :=\det(\gamma_{\alpha\beta}),$$
$$\Gamma_{\beta\eta}^\alpha=\frac{1}{2}\gamma^{\alpha\delta}(\gamma_{\beta\delta,\eta}+\gamma_{\eta\delta,\beta}-\gamma_{\beta\eta, \delta})$$
and similarly 
$$g^{ij}, \Gamma^i_{jk}.$$
If $fM\to N$ is a map of $c^1$ and is said to be harmonic if it satisfies, in local coordinates $x=(x^1, ..., x^m) \in M$
$$ \frac{1}{\sqrt{\gamma}}\frac{\partial}{\partial x^\alpha}(\sqrt{\gamma}\gamma^{\alpha\beta}\frac{\partial}{\partial x^\beta}f^i)+\gamma^{\alpha\beta}(x)\Gamma^i_{jk}(f(x))\frac{\partial}{\partial x^\alpha}f^j\frac{\partial}{\partial x^\beta}f^k=0.$$

Theorem A with initial values implies the existence of a local harmonic map with given tangent 
plane at $q$.

\bigskip

School of Mathemtics and Informatics,

Jiangxi Normal University, Nanchang, China

\bigskip

Department of Mathematical Sciences

Indiana University - Purdue University Fort Wayne

Fort Wayne, IN 46805-1499, USA.

pan@ipfw.edu
\end{document}